\documentclass{article} 
\usepackage{url}  
\usepackage{graphicx} 
\usepackage{subfig}

\usepackage{amsmath}
\usepackage{amsthm}
\usepackage{amsfonts}
\usepackage{amssymb}
\usepackage{nccmath}
\usepackage{url}
\usepackage[authoryear]{natbib}

\usepackage{algorithm}
\usepackage{algorithmic}

\newtheorem{thm}{Theorem}%[section]

\newtheorem{lem}[thm]{Lemma}
\newtheorem{example}[thm]{Example}

\newcommand{\RR}{ \mathbb{R} }

\newcommand{\spaceo}{\hspace{2 mm}}
\newcommand{\setsep}{ \spaceo | \spaceo}
\newcommand{\half}{\frac{1}{2}}

\DeclareMathOperator*{\Id}{Id}
\DeclareMathOperator*{\spanop}{span}
\DeclareMathOperator*{\diag}{diag}

\newcommand{\Abs}[1]{\left| #1 \right|}
\newcommand{\Set}[1]{\left\{ #1 \right\}}
\newcommand{\Brack}[1]{\left( #1 \right)}
\newcommand{\sqBrack}[1]{\left[ #1 \right]}
\newcommand{\inner}[2]{\left< #1 , #2 \right>}
\newcommand{\sqinner}[2]{\left[ #1 , #2 \right]}
\newcommand{\dinner}[2]{\left<\left< #1 , #2 \right>\right>}

\newcommand{\Exp}[1]{ \mathbb{E} #1}

\newcommand{\eps}{\varepsilon}

\title{On-Line Learning of Linear Dynamical Systems: \\ Exponential Forgetting in Kalman Filters}
\author{Mark Kozdoba, Jakub Marecek, \\ Tigran Tchrakian, and Shie Mannor} 
\date{}

\begin{document}
\maketitle

\begin{abstract}
The Kalman filter is a key tool for time-series forecasting and analysis.
We show that the dependence of a prediction of Kalman filter on the past is decaying exponentially, whenever the process noise is non-degenerate. 
Therefore, Kalman filter may be approximated by regression on a few 
recent observations.
Surprisingly, we also show that having some process noise is \textit{essential} 
for the exponential decay. 
With no process noise, it may happen that the forecast depends on all of the past uniformly, which makes forecasting more difficult.

Based on this insight, we devise an on-line algorithm for improper learning of a linear dynamical system (LDS), which considers only a few most recent observations. 
We use our decay results to provide the first regret bounds w.r.t.~to Kalman filters within  learning an LDS. That is, we compare the results of our algorithm to the best, in hindsight, Kalman filter for a given signal.  Also, the algorithm is practical: its per-update run-time is linear in the regression depth. 
%In addition, we show that the algorithm outperforms other recently proposed approaches on a few well-known examples.
\end{abstract}

%We stress that while there were algorithms with regret bounds previously, in no case has the regret considered Kalman filtering,
%despite its practical relevance. 

\section{Introduction}
Linear Dynamical Systems (LDS) are a key standard tool in modeling and 
forecasting time series, with an exceedingly large number of applications.  
In forecasting with an LDS, typically one learns the parameters of the LDS 
first, using a maximum likelihood principle, and then uses Kalman 
filter to generate predictions. The two features that seem to contribute 
the most to the success of LDS in practice are the ability of 
LDS to model a wide range of behaviors, and the recursive nature of 
Kalman filter, which allows for fast, real-time forecasts via
a constant-time update of the previous estimate. 
On the other hand, a major difficulty with LDSs is that the process of 
learning system parameters, via expectation maximization (EM) or direct 
likelihood optimization, may be time consuming and prone to getting stuck in local maxima. 
We refer to 
\cite{anderson1979,WestHarrison,hamilton1994time,chui2017kalman} 
for book-length introductions.

Recently, there has been an interest in alternative, 
\textit{improper} learning approaches, where one approximates the 
predictions of LDSs by a linear function of 
a few past observations. The advantage of such approaches is that it 
convexifies the problem, i.e., learning the linear function amounts to a 
convex problem, which avoids the issues brought by the non-convex nature 
of the likelihood function. The convexification allows for on-line 
algorithms, which are typically fast and simple. A crucial advance of 
these recent approaches is the guarantee that  the predictions 
of the convexified, improper-learning algorithm are at least as good as the 
predictions of the proper one. One therefore avoids the long learning 
times and issues related to non-convexity associated with the classical algorithms, while 
maintaining the statistical performance. 

Leading examples of this approach  
\cite{anava13,arima_aaai,hazan2017online} 
utilise a framework 
of \textit{regret bounds} \cite{cesa2006prediction}
to provide guarantees on the performance of the convexifications. 
In this framework, one considers a sequence of observations $Y_t$, with or 
without additional assumptions. After observing $Y_0,\ldots,Y_t$, an 
algorithm  for improper learning produces a forecast $\hat{Y}_{t+1}$ of 
the next observation. 
Then, roughly speaking, 
one shows that the sum of errors of the forecast thus produced 
is close to the sum of errors of the best model (in hindsight)  
from within a certain class. It is said that the algorithm competes 
against a certain class. 

In this paper, we take several steps towards developing guarantees for an 
algorithm, which competes against Kalman filters. 
Specifically, we ask what conditions make it possible to model the 
predictions of Kalman filter 
as a regression of a few past observations? 
We show that for a natural, large, and well-known 
class of LDSs, the \textit{observable} LDSs, 
the dependence of Kalman filter on the past 
decays exponentially 
if the process noise of the LDS is 
non-degenerate. 
Consequently, predictions of such LDS can be modeled 
as auto-regressions. 
In addition, we show that at least some non-degeneracy
of the process noise is \textit{necessary} for the exponential decay. We provide an example with no process noise, where the dependence on the past 
does not converge exponentially. 

Next, based on the decay results, we give an on-line
algorithm for time-series prediction and prove regret bounds for it. 
The algorithm makes predictions in the form 
$\hat{y}_{t+1} = \sum_{i=0}^{s-1} \theta_i(t) Y_{t-i}$, where $Y_{t}$ 
are observations, and $\theta(t)\in\RR^s$ is the 
vector of auto-regression (AR) coefficients, which is updated by the algorithm in an on-line manner. 

For any LDS $L$, denote by $f_{L,t+1}$ the predicted value of $Y_{t+1}$ by 
Kalman filter corresponding to $L$, given $Y_t, \ldots, Y_0$. 
Denote by $E(T) = \sum_{t=1}^{T-1} \Abs{\hat{y}_{t+1}-Y_{t+1}}^2$
the total error made by the algorithm up to time $T$, and by 
$$E(L,T) = \sum_{t=1}^{T-1} \Abs{f_{L,t+1}-Y_{t+1}}^2$$ the total 
error made by Kalman filter corresponding to $L$. 
Let $S$ be any finite family of observable linear dynamical systems with non-degenerate process noise. We show that for an appropriate regression depth $s$, 
for any bounded sequence $\Set{Y_t}_{0}^T$ we have
\begin{equation}
 \frac{1}{T}E(T) \leq \frac{1}{T} \min_{L \in S} E(L,T) + 
 \frac{1}{T}C_{S} + \eps, 
\end{equation}
where $C_{S}$ is a constant depending on the family $S$. 
In words, up to an arbitrarily small $\eps$ given in advance, the average prediction error of the algorithm is as good or 
better than the average prediction error of the \textit{best} Kalman 
filter in $S$. We emphasize that while there is a dependence on $S$ in the 
bounds, via the constant $C_S$, the algorithm itself depends on $S$ 
only through the regression depth $s$. In particular, the algorithm does 
not depend on the cardinality of $S$, and the time complexity of each 
iteration is $O(s)$.

To summarize, our contributions are as follows: 
We show that the dependence of predictions of Kalman filters in a 
system with non-degenerate process noise is exponentially decaying and that  
therefore Kalman filters may be approximated by regressions on a 
few recent observations, cf. Theorem \ref{thm:lds_approximation}.   
We also show that the process noise is \textit{essential} for the 
exponential decay. We given an on-line prediction algorithm and prove the 
first regret bounds against Kalman filters, cf. Theorem 
\ref{thm:regret_bound}.   
Experimentally, we illustrate the performance on a single example in the 
main body of the text, and further examples in the supplementary material.

\section{Literature}
\label{sec:literature}

%Jakub: AAAI has weird rules about ``unpublished work''.
%They assume that everyone knows about all published work, but no one 
%except the authors knows about the unpublished work.  
%Generally, no unpublished work should be cited. It's not clear whether 
%this includes Hazan 2018 (in arxiv), but I'd rather skip it.

In this section, we review the relevant literature and place the current 
work in context. 

We refer to \cite{hamilton1994time} for an exposition on LDSs, Kalman 
filter, and the classical approach to learning the LDS parameters via tha 
maximum likelihood optimization. See also \cite{RG1999} for a survey 
of relations between LDSs and a large variety of other probabilistic 
models. 
A general exposition of on-line learning can be found in 
\cite{Hazanbook2016}. 

As discussed in the Introduction, we are concerned with improper learning, where we show that an alternative model can be shown to generate forecasts that are as good as Kalman filter, up to any given error. Perhaps the first example of an 
improper learning that is still used today is the 
moving average, or the exponential moving average 
\cite{gardner1985exponential}. In this approach,
predictions for a process -- of a possibly complex nature -- are made 
using a simple auto-regressive (AR) or AR-like model. 
This is very successful in a multitude of engineering applications.
Nevertheless, until recently, there were very few guarantees for the 
performance of such methods. 

In \cite{anava13},  the first guarantees regarding 
prediction of a (non-AR) subset of auto-regressive moving-average (ARMA) processes by AR processes were given, together with an algorithm for finding the 
appropriate AR. In \cite{arima_aaai}, these results were extended to 
a subset of autoregressive integrated moving average (ARIMA) processes, while at the same time the assumptions on the underlying ARMA model were relaxed. 

% As discussed in the Introduction, 
In this paper, we show that AR models 
may also be used to forecast as well as Kalman filters. 
One major difference between our results and the previous work is that we 
obtain approximation results on \textit{arbitrary} bounded sequences.
Indeed, regret results of \cite{anava13} and \cite{arima_aaai} only hold 
under the assumption that the data sequence was generated by a
particular fixed ARMA or ARIMA process.  Moreover, the constants in the 
regret bounds of \cite{anava13} and \cite{arima_aaai} depend on the generating model, and the guaranteed convergence  may be arbitrarily slow, even when the sequence to forecast is generated by appropriate model. 

In contrast, we show that up to an arbitrarily small error given in 
advance, AR($s$) will perform as well as Kalman filter on any bounded 
sequence. We also obtain approximation results in the more general 
case of bounded difference sequences.

Another related work is \cite{hazan2017online}, which addresses a 
different aspect of LDS approximation by ARs. In the case of LDSs with 
\textit{inputs}, building on known eigenvalue-decay estimates of Hankel 
matrices, it is shown that the influence of all past inputs 
may be effectively approximated by an AR-type model.
However, the arguments and the algorithms in \cite{hazan2017online} were 
not designed to address model noise. In particular, 
the algorithm of \cite{hazan2017online} makes predictions based on the 
whole history of inputs and on only one most recent observation, $Y_t$, 
and hence clearly can not compete with Kalman filters in situations with 
no inputs. We demonstrate this in the Experiments section. 

\label{runtimediscussion}
Also, note that while on-line gradient descent (OGD) is used throughout much recent work in forecasting  
\cite{anava13,hazan2017online,arima_aaai}, there are important differences in the per-iteration run-time, as well as the problems solved. For example, the OGD of \cite{hazan2017online} requires $O(t^3)$ operations to forecast at time $t$, as it does consider the whole history of $t$ inputs and an eigen-decomposition of an $t \times t$ matrix at each time $t$, although this could be pre-computed. 
Our Algorithm \ref{alg:OGD} has a liner per-iteration complexity in the regression depth $s$, which can be as small as  $s = 1$. Notably, its run-time at time $t$ does not depend on time $t$.
%We refer to Figure~\ref{fig3} in the section on our Experiements for a graphical illustration. 

\section{Preliminaries} % Kalman Filters,
\label{sec:definitions}

As usual in the literature \cite{WestHarrison}, we define a linear system $L = (G,F,v,W)$ as:

\begin{eqnarray}
\phi_{t} = G h_{t-1} + \omega_t \\
Y_t = F' \phi_t + \nu_t,
\end{eqnarray}

where $Y_t$ are scalar observations, and $\phi_t \in \RR^{n\times1}$
is the hidden state. $G \in \RR^{n\times n}$ is the state transition 
matrix which defines the system dynamics, and $F \in \RR^{n\times1}$ 
is the observation direction. The process-noise terms $\omega_t$ 
and observation-noise terms $\nu_t$ are mean zero normal independent 
variables. For all $t\geq 1$ the covariance of $\omega_t$ is 
$W$ and the variance of $\nu_t$ is $v$. The initial state $\phi_0$ is a 
normal random variable with mean $m_0$ and covariance $C_0$.

For $t\geq 1$ denote 
\begin{equation}
m_t = \Exp{\Brack{\phi_t | Y_0,\ldots,Y_t}}, 
\end{equation}
and let $C_t$ be the covariance matrix of $\phi_t$ given 
$Y_0,\ldots,Y_t$. Note that $m_t$ is the estimate of the current hidden 
state, given the observations. Further, the central quantity of this paper
is 
\begin{equation}
f_{t+1} = \Exp{\Brack{Y_{t+1}|Y_t, \ldots,Y_0}} = F'G m_t.
\end{equation}
This is the forecast of the next observation, given the current data. 
The quantities $m_t$ and $f_{t+1}$ are known as Kalman Filter. 
In particular, in this paper we refer to the sequence $f_{t}$ as the 
Kalman filter associated with the LDS $L=(G,F,v,W)$. 

The Kalman filter satisfies the following recursive update equations: 
Set
\begin{eqnarray*}
a_t &=& G m_{t-1} \\
R_t &=& G C_{t-1} G' + W \\ % R_{t+1}  = W + G*(R_t -  A\times A *Q)*G'
Q_t &=& F'R_tF + v \\
A_t &=& R_t F  / Q_t
\end{eqnarray*}
Note that in this notation we have
\begin{equation*}
f_t = F' a_t. 
\end{equation*}
Then the update equations of Kalman filter are:
\begin{eqnarray}
m_t &=& a_t + A_t(Y_t - f_t) = A_t Y_t + (I - F \otimes A_t) a_t \label{eq:m_t_update}\\
C_t &=& R_t - A_t Q_t A'_t
\end{eqnarray}
where $x\otimes y$ is an $\RR^{n\times1} \rightarrow \RR^{n\times1}$ 
operator which acts by $z \mapsto \inner{z}{x}y = yx'z$. The matrix of 
$x\otimes y$ is given by the outer product $yx'$, where $x,y \in \RR^{n\times1}$.

An important property of Kalman Filter is 
that while $m_t$ depends on $Y_0,\ldots,Y_t$, the covariance matrix 
$C_t$ does not. Indeed, note that $R_t,Q_t,A_t,C_t$ are all deterministic 
sequences which do not depend on the observations.

We explicitly write the recurrence relation for $R_t$: 
\begin{equation}
\label{eq:R_recursion}
R_{t+1} = G\Brack{R_t - \frac{R_t F \otimes R_t F}{\inner{F}{R_t F} + v}}G' + W
\end{equation}
Also write for convenience 
\begin{equation}
\label{eq:a_t_update}
a_{t+1} = Gm_{t} = GA_t Y_t + G(I-F\otimes A_t)a_t. 
\end{equation}

A more explicit form of the prediction of $Y_{t+1}$ given 
$Y_t, \ldots,Y_0$, may be obtained by 
unrolling (\ref{eq:m_t_update}) and using (\ref{eq:a_t_update}):
\begin{eqnarray}
\Exp{\Brack{Y_{t+1}|Y_t, \ldots,Y_0}}   &=& f_{t+1} = F'a_{t+1} \quad \quad \label{thatexpintext}
\end{eqnarray}
\begin{eqnarray}
\quad  &=& F'GA_tY_t + F'G(I-F\otimes A_t)a_t  \\
  &=& F'GA_tY_t + F'G(I-F\otimes A_t)GA_{t-1}Y_{t-1} \nonumber\\
  &&    + F'G(I-F\otimes A_t)G(I-F\otimes A_{t-1})a_{t-1}.
\end{eqnarray}

In general, set $Z_t = G(I-F\otimes A_t)$ and $Z = G(I-F \otimes A)$. 
Chose and fix some $s\geq 1$. Then for any $t\geq s+1$, 
the expectation \eqref{thatexpintext} has the form displayed in 
Figure \ref{fig:exp}.

%\begin{figure*}
%\begin{equation}
%\label{eq:base_prediction}
%\Exp{\Brack{Y_{t+1}|Y_t, \ldots,Y_0}} %= 
%f_{t+1} = 
%\underbrace{  F'GA_tY_t + 
%F' \sum_{j=0}^{s-1} %\Brack{\prod_{i=0}^j Z_{t-i}}GA_{t-%j}Y_{t-j} % -1 SO THERE COULD BE THE %BUG?! 
%}_{AR(s+1)}
%          + 
%\underbrace{          
%          F' \Brack{\prod_{i=0}^s %Z_{t-i}} a_{t-s}.}_{
%          \text{Remainder term}} %by Theorem}~\ref{thm:main_contraction} %}.
%\end{equation}
%\caption{The unrolling of the %forecast $\Exp{\Brack{Y_{t+1}|Y_t, %\ldots,Y_0}}$.
%Notice that the remainder term goes %to zero exponentially fast 
% with $s$, as per %Theorem~\ref{thm:main_contraction}.
%}
%\label{fig:exp}
%\end{figure*}

\begin{figure*}
\begin{equation}
\label{eq:base_prediction}
f_{t+1} = 
\underbrace{  F'GA_tY_t + 
F' \sum_{j=0}^{s-1} \sqBrack{ \Brack{\prod_{i=0}^j Z_{t-i}}GA_{t-j-1}Y_{t-j-1}} % -1 SO THERE COULD BE THE BUG?! 
}_{AR(s+1)}
          + 
\underbrace{          
          F' \Brack{\prod_{i=0}^s Z_{t-i}} a_{t-s}.}_{
          \text{Remainder term}} %by Theorem}~\ref{thm:main_contraction} }.
\end{equation}
\caption{The unrolling of the forecast $f_{t+1}$.
The remainder term goes to zero exponentially fast 
 with $s$, by Lemma \ref{lem:state_magnitude_bound}.
}
\label{fig:exp}
\end{figure*}

Next, a linear system $L=(G,F,v,W)$ is said to be \textit{observable}, 
\cite{WestHarrison}, if
\begin{equation}
\spanop\Set{F, G'F,\ldots,{G'}^{n-1}F} = \RR^n.
\end{equation}
Roughly speaking, the pair $(G,F)$ is observable if the state can be 
recovered from a sufficient number of observations, in a noiseless situation. Note that 
if there were parts of the state that do not influence the observations, 
these parts would be irrelevant for forecast purposes. Thus we are only 
interested in observable LDSs. 

When $L$ is observable, it is known \cite{harrison1997convergence} that the sequences 
$C_t,R_t,Q_t,A_t$ converge. See also \cite{anderson1979,WestHarrison}.
We denote the limits by $C,R,Q$ and $A$ 
respectively. Moreover, the limits satisfy the recursions as equalities. In particular we have
\begin{equation}
\label{eq:R_equality}
R = G\Brack{R - \frac{RF \otimes RF}{\inner{F}{RF} + v}}G' + W. 
\end{equation}

Finally, an operator $P: \RR^n \rightarrow \RR^n$ is 
\textit{non-negative}, denoted $L \geq 0$, if $\inner{Px}{x} \geq 0$ for 
all $x \neq 0$, and is \textit{positive}, denoted $P>0$,  if 
$\inner{Px}{x} > 0$ for all $x \neq 0$. Note that $W,C_t,R_t,C,R$ are 
either covariance matrices or limits of such matrices, and thus are 
symmetric and non-negative.

\section{Exponential Decay and AR Approximation}
\label{sec:analysis}

In what follows, we denote by 
\begin{equation}
\sqinner{x}{y} = \inner{Rx}{y}, \spaceo \dinner{x}{y} = \inner{Wx}{y}
\end{equation}
the inner products induced by $R$ and $W$ on $\RR^n$, where $R$ is the limit of $R_t$ as described above. 
In particular, we set $U=G'$ and rewrite (\ref{eq:R_equality}) as 
\begin{equation}
\label{eq:R_inner_eq}
\sqinner{x}{y} = \sqinner{Ux}{Uy} - 
                 \frac{\sqinner{Ux}{F}{\sqinner{Uy}{F}}}{\sqinner{F}{F}+v} 
                 + \dinner{x}{y}.
\end{equation}

Observe that since $R = GCG' + W$, we have $R \geq W$, and in particular 
if $W>0$ then $R>0$. In other words, if $W>0$, then $\sqinner{\cdot}{\cdot}$ and $\dinner{\cdot}{\cdot}$ induce proper norms on $\RR^n$:
\begin{equation}
\label{eq:R_W_order_inner}
\sqinner{x}{x} \geq \dinner{x}{x} > 0 \mbox{ for all $x\neq 0$}.
\end{equation}

Next, consider the remainder term in the prediction equation 
(\ref{eq:base_prediction}), where we have replaced $Z_{t-i}$ with their limit values $Z$:
\begin{eqnarray*}
F'& \Brack{G(I - F \otimes A)}^{s+1} a_{t-s} \quad\quad\quad \\
& = 
\inner{F}{\Brack{G(I - F \otimes A)}^{s+1} a_{t-s}} \\ 
& = \inner{\Brack{(I - A \otimes F)U}^{s+1} F}{a_{t-s}}. 
\end{eqnarray*}

Let us now state and prove the key result of this paper: if $W>0$, then 
$\Brack{(I - A \otimes F)U}^s F$ converges to zero exponentially fast 
with $s$. The key to the proof will be to consider contractivity 
properties with respect to the norm induced by $\sqinner{\cdot}{\cdot}$, 
rather than with respect to the the default inner product. 

\begin{thm} 
\label{thm:main_contraction}
If $W>0$, then there is \newline $\gamma = \gamma(W,v,F,G) < 1$ 
 such that for every $x \in \RR^n$,  
\begin{equation}
\sqinner{(I - A \otimes F)Ux}{(I - A \otimes F)Ux} \leq \gamma \sqinner{x}{x}. 
\end{equation}
\end{thm}
\begin{proof}
Set 
\begin{equation}
y = \Brack{(I - A \otimes F)U}x. 
\end{equation}
Then
\begin{align}
\label{eq:Z_tag_action}
y &  = (I - A \otimes F)U x = Ux - \inner{A}{Ux}F\\ 
& = Ux - \frac{\sqinner{Ux}{F}}{\sqinner{F}{F}+v} F. 
\end{align}

Therefore we have
\begin{equation}
\label{eq:v_k_norm_rec}
\sqinner{y}{y} = \sqinner{Ux}{Ux} 
-2\frac{\sqinner{Ux}{F}^2}{\sqinner{F}{F}+v} + 
\frac{\sqinner{Ux}{F}^2 \sqinner{F}{F}}{\Brack{\sqinner{F}{F}+v}^2}. 
\end{equation}

In addition, by (\ref{eq:R_inner_eq}), 
\begin{equation}
\label{eq:v_k_norm_U_ricatti}
\sqinner{Ux}{Ux} = \sqinner{x}{x} + \frac{\sqinner{Ux}{F}^2}{\sqinner{F}{F}+v} - \dinner{x}{x}. 
\end{equation}
Combining (\ref{eq:v_k_norm_rec}) and (\ref{eq:v_k_norm_U_ricatti}), we obtain
\begin{align}
\label{eq:v_k_base_contractivity}
\sqinner{y}{y} =&  
\sqinner{x}{x} - \dinner{x}{x}
-\frac{\sqinner{Ux}{F}^2}{\sqinner{F}{F}+v}
\left( 1 - \frac{\sqinner{F}{F}}{\sqinner{F}{F}+v} \right) \notag \\
=& \sqinner{x}{x} - \dinner{x}{x}
-\frac{\sqinner{Ux}{F}^2}{\sqinner{F}{F}+v}
\frac{v}{\sqinner{F}{F}+v}. \end{align}
Equation (\ref{eq:v_k_base_contractivity}) immediately implies that 
$[x,x]$ is non-increasing. 
Recall that by (\ref{eq:R_W_order_inner}), $W$ is dominated by $R$. 
However, since both $R$ and $W$ define proper norms, by the equivalence of 
finite dimensional norms, the inverse inequality is also true: There 
exists $0<\kappa\leq 1$ such that 
\begin{equation}
\label{eq:inverse_R_W_equiv}
\dinner{x}{x} \geq \kappa \sqinner{x}{x} \mbox{ for all $x\neq 0$}.
\end{equation}
Therefore the decrease in (\ref{eq:v_k_base_contractivity}) must be 
exponential:
\begin{equation}
\sqinner{y}{y} \leq 
\sqinner{x}{x} - \dinner{x}{x} \leq 
(1 - \kappa)\sqinner{x}{x}.
\end{equation}
\end{proof}

It is of interest to stress the fact that Theorem \ref{thm:main_contraction} does 
not assume any contractivity properties of $G$. In particular, 
the very common assumption of the spectral radius of $G$ being bounded 
by $1$ is not required. 

Let us state and prove our main approximation result: 
\begin{thm}[LDS Approximation]
\label{thm:lds_approximation}
Let $L=L(F,G,v,W)$ be an observable LDS with $W>0$. 
\begin{enumerate}
\item 
For any $\eps >0$, and any $B_0 >0$, there is 
$T_0>0$, $s>0$ and $\theta \in \RR^s$, such that 
for every sequence $Y_t$ with $\Abs{Y_t}\leq B_0$,
and for every $t \geq T_0$, 
\begin{equation}
\label{eq:bounded_approx_in_thm_approx}
\Abs{f_{t+1} - \sum_{i=0}^{s-1} \theta_i Y_{t-i}} \leq \eps.
\end{equation}
\item 
For any $\eps,\delta >0$, and any $B_1 >0$, there is 
$T_0>0$, $s>0$ and $\theta \in \RR^s$, such that 
for every sequence $Y_t$ with $\Abs{Y_{t+1}-Y_t}\leq B_1$,
and for every $t \geq T_0$, 
\begin{equation}
\Abs{f_{t+1} - \sum_{i=0}^{s-1} \theta_i Y_{t-i}} \leq 
2\max\Brack{\eps, \delta \Abs{Y_t}}.
\end{equation}
\end{enumerate}  
\end{thm}

We first prove the bound on the remainder term in the prediction equation
(\ref{eq:base_prediction}).

\begin{lem}[Remainder-Term Bound]
\label{lem:state_magnitude_bound}
Let $L=L(F,G,v,W)$ be an observable LDS with $W>0$. 
\begin{enumerate}
\item
If a sequence $Y_t$ satisfies $\Abs{Y_t}\leq B_0$ for all $t\geq 0$, then 
there are constants $\rho_L'<1$ and $c_L$ such that  for any $s>0$ and $t > s$,
\begin{equation}
\Abs{\inner{F}{\Brack{\prod_{i=0}^s Z_{t-i}} a_{t-s}}} \leq (\rho_L')^s c_L.
\end{equation}
\item 
If a sequence $Y_t$ satisfies $\Abs{Y_{t+1} - Y_t}\leq B_1$ for 
all $t\geq 0$, then there are constants $\rho'_L$ and $c_{1,L},c_{2,L}$ such that for all $s>0$ and $t > s$, 
\begin{equation}
\Abs{\inner{F}{\Brack{\prod_{i=0}^s Z_{t-i}} a_{t-s}}} \leq  (\rho_L')^s c_{1,L}\Brack{\Abs{Y_{t}} + sB_1 + c_{2,L}}. 
\end{equation}
\end{enumerate}
\end{lem}
\begin{proof} 
Recall that $a_t$ satisfies the recursion (\ref{eq:a_t_update}),
\begin{equation}
\label{eq:lem_at_rec}
a_{t+1} = G(I - F \otimes A_t)a_t + Y_t A_t = Z_ta_t + Y_t G A_t. 
\end{equation}
Denote by $[x] = \sqinner{x}{x}^{\half}$ and 
and by $\Abs{x} = \inner{x}{x}^{\half}$
the norms induced by  
$\sqinner{\cdot}{\cdot}$ and $\inner{\cdot}{\cdot}$ respectively. 
Set $P = Z'$ and $P_t = Z_t'$.
By Theorem \ref{thm:main_contraction}, there is 
$\rho = \gamma^{\half}<1$ such that $P$ is a $\rho$-contraction with respect to $[\cdot]$. Fix some $\rho'$ such that $\rho<\rho'<1$. 
Since $P_t \rightarrow P$, there is some $T_1$ such that for all 
$t\geq T_1$, $P_t$ is a $\rho'$-contraction. In addition, let $T_2$ be 
such that $[GA-GA_t]\leq 1$ for all $t\geq T_2$. Set 
$T_0=\max\Brack{T_1,T_2}+1$.  Fix $s>0$ and set $t' = t-s-1$.
For $t'>T_0$, using (\ref{eq:lem_at_rec}) write $a_{t-s}$ as 
\begin{align}
% &= \\\
\label{eq:thm_approx_ats_unroll1}
a_{t'+1} &= Y_{t'} GA_{t'} + 
\sum_{i=0}^{t'-T_0} \Brack{Y_{t'-i-1} \Brack{\prod_{j=0}^i Z_{t'-j}} GA_{t'-i-1} }
\nonumber \\
& + \Brack{\prod_{j=0}^{t'-T_0} Z_{t'-j}} a_{T_0-1}. 
\end{align}

Observe that if an operator $O'$ is a $\gamma$-contraction with respect to 
$[\cdot]$, then for any $x,y\in \RR^n$, 
\begin{align}
\label{eq:lem_norm_and_contract_inner}
\inner{y}{Ox} &= \inner{O'y}{x} \\ 
              &  \leq \Abs{O'y}\Abs{x} %\nonumber \\%
                \leq \gamma \mu [y]\Abs{x} \leq \gamma \mu^2 [y][x], \nonumber
\end{align}
where $\mu$ is the equivalence constant between $[\cdot]$ and $\Abs{\cdot}$.

For every $x \in \RR^n$ by (\ref{eq:thm_approx_ats_unroll1}) we have
\begin{align}
\label{eq:x_ats_inner}
\inner{x}{a_{t-s}} &=  \\
 &= Y_{t'} \inner{x}{GA_{t'}} +  \nonumber \\
 &   \sum_{i=0}^{t'-T_0} 
 Y_{t'-i-1} \inner{ \Brack{\prod_{j=i}^0 P_{t'-j}}x}{GA_{t'-i-1}} 
  \nonumber \\
& + \inner{\Brack{\prod_{j=t'-T_0}^{0} P_{t'-j}}x}{a_{T_0-1}}.  \nonumber
\end{align}
By the choice of $T_0$, as since the expansion in (\ref{eq:x_ats_inner}) is only up to $T_0$, every $P_{t'-j}$ in (\ref{eq:x_ats_inner}) is a 
$\rho'$-contraction and all $GA_{t'-j}$ satisfy $[GA-GA_{t'-j}]\leq 1$. 

Combining this with (\ref{eq:lem_norm_and_contract_inner}) and using triangle inequality, we obtain 
\begin{align}
\Abs{\inner{x}{a_{t-s}}} &\leq  \\
 &= \Abs{Y_{t'}} \mu^2 [x]\Brack{[GA]+1} +  \nonumber \\
 & + \sum_{i=0}^{t'-T_0} 
 \Abs{Y_{t'-i-1}} (\rho')^{i+1} \mu^2 [x]\Brack{[GA]+1}
  \nonumber \\
& +  (\rho')^{t'-T_0}\mu^2 [x][a_{T_0-1}].  \nonumber
\end{align}
Finally, choose $x = \Brack{\prod_{i=s}^0 P_{t-i}} F$. Note that 
$[x] \leq (\rho')^{s+1}[F]$.  Therefore, 
\begin{align}
\label{eq:lem_long_stuff_begin}
&\Abs{\inner{F}{\Brack{\prod_{i=0}^s Z_{t-i}} a_{t-s}}} = 
\inner{x}{a_{t-s}}   \\
& \leq (\rho')^{s+1} \Abs{Y_{t'}} \mu^2 [F]\Brack{[GA]+1} +  \\
&  + \sum_{i=0}^{t'-T_0} 
 (\rho')^{s+1}\Abs{Y_{t'-i-1}} (\rho')^{i+1} \mu^2 [F]\Brack{[GA]+1}
   \label{eq:lem_long_stuff_series}\\
& +  (\rho')^{s+1} (\rho')^{t'-T_0}\mu^2 [F][a_{T_0-1}].  
   \label{eq:lem_long_stuff_const_term}
\end{align}
Observe that the term $[a_{T_0-1}]$ in (\ref{eq:lem_long_stuff_const_term}) 
is a constant, independent of $t$, and that the series in 
(\ref{eq:lem_long_stuff_series}) are summable w.r.t $t'$. Therefore, in the bounded 
case $\Abs{Y_t} \leq B_0$, the proof is complete. 

In the Lipschitz case, for every $i>0$, we have 
\begin{equation}
\Abs{Y_{t'-i-1}} \leq \Abs{Y_{t'}} + (i+1)B_1. 
\end{equation}
Substituting this into 
(\ref{eq:lem_long_stuff_begin})-(\ref{eq:lem_long_stuff_const_term}), and observing that the resulting 
series are still summable, we obtain 
\begin{equation}
\Abs{\inner{F}{\Brack{\prod_{i=0}^s Z_{t-i}} a_{t-s}}} \leq  
(\rho')^s c_{1}\Brack{\Abs{Y_{t'}} + c_{2}}. 
\end{equation}
Thus using 
\begin{equation}
\Abs{Y_{t'}} \leq \Abs{Y_{t}} + sB_1, 
\end{equation}
completes the proof in the Lipschitz case. 
\end{proof}

We now prove Theorem \ref{thm:lds_approximation}. 
\begin{proof}
Recall that $f_{t+1}$ is given by (\ref{eq:base_prediction}). 
Fix some $s>0$ and set $\theta_0 = \inner{F}{GA}$, 
and $\theta_{j+1} = \inner{F}{Z^{j+1}GA}$ for $j=0,\ldots,s-1$. 
Note that $\theta \in \RR^{s+1}$ and $s$ here corresponds to $s+1$ in the 
statement of the Theorem. 
Set also $r_t = \inner{F}{GA_t}$ and for $j\geq 0$, 
$r_{t-j-1} = \inner{F}{\Brack{\prod_{i=0}^j Z_{t-i}}GA_{t-j}}$.
Clearly $r_t \rightarrow \theta_0$ with $t$ and 
$r_{t-j-1} \rightarrow \theta_{j+1}$ for every fixed $j$. 
Next, using Lemma \ref{lem:state_magnitude_bound}, the discrepancy between $f_{t+1}$ and the $\theta$ predictor is given by 
\begin{align}
\label{eq:lds_approx_thm_abs_bound_bounded_case}
&\Abs{f_{t+1} - \sum_{j=0}^s Y_{t-j} \theta_j} \leq  \\
&\Abs{Y_t} \Abs{r_t - \theta_0} + \sum_{j=0}^{s-1} \Abs{Y_{t-j-1}} \Abs{r_{t-j-1} - \theta_{j+1}} 
+ (\rho_L')^s c_L  \nonumber
\end{align}
in the bounded case. In this case, therefore, choosing regression depth $s$ large enough 
so that $(\rho_L')^s c_L \leq \eps/2$ and $T_0$ large enough so that 
for all $t\geq T_0$, $\Abs{ r_{t-j-1} - \theta_{j+1} } \leq \frac{\eps}{2sB_0}$ for all $j\leq s$, suffices to conclude the proof. 
The proof of the Lipschitz case follows similar lines and is given in the Supplementary Material due to space constraints. 
\end{proof}

To conclude this section, we discuss the relation between exponential 
convergence and the non-degenerate noise assumption, $W>0$. 
Note that the crucial part of Theorem \ref{thm:main_contraction}, 
inequality (\ref{eq:inverse_R_W_equiv}), holds if and only if we can 
guarantee that $\dinner{x}{x}>0$ for every $x$ for which 
$\sqinner{x}{x}>0$. In particular, this holds when $W>0$ -- that is, 
the noise is full dimensional. We now demonstrate that at least some 
noise is \textit{necessary} for the exponential decay to hold.

Consider first a one dimensional example. 
\begin{example}
With $n=1$, assume that $Y_t$ are generated by an LDS with $G=F=1$, $W=0$ 
and some $v>0$. Assume that the true process starts from a deterministic 
state $m_{0,ture} >0$. Since we do not know $m_{0,true}$, we start the 
Kalman filter with $m_0 = 0$ and initial covariance $C_0 = 1$. 
\end{example}
In this case, clearly the observations $Y_t$ are independent samples
of a fixed distribution with mean $m_{0,true}$ and variance $v$. 
The Kalman filter in this situation is equivalent to a Bayesian 
mean estimator with prior distribution $N(0,C_0=1)$. From general 
considerations, it follows that $R_t \rightarrow R = 0$ with $t$. Indeed, 
if we start with $C_0 = 0$, then we have $R_t = 0$ for all $t$. 
Since the limit $R$ does not depend on the initialization, 
\cite{harrison1997convergence}, we have $R=0$ for every initialization. 
As a side note, in this particular case it can be shown, either via 
the Bayesian interpretation or directly, that $R_t$ decays as $1/t$ 
(that is, $t R_t \rightarrow const$, with $t$). 
Now, note that $Z_t = 1 - \frac{R_t}{R_t + v} = \frac{v}{R_t + v} \rightarrow 1$, 
and that for any fixed $j>0$, $A_{t-j} \rightarrow 0$ as $t$ grows. 
Next, for fixed $s>0$, consider the prediction equation  
(\ref{eq:base_prediction}). On the one hand, we know that $f_{t+1}$ 
converges to $m_{0,true}>0$ in probability. This is clear for instance 
from the Bayesian estimator interpretation above. On the other hand,
the coefficients of all $Y_{t-j}$ in (\ref{eq:base_prediction}) converge 
to $0$. It follows therefore, that the remainder term in 
(\ref{eq:base_prediction}), $F' \Brack{\prod_{i=0}^s Z_{t-i}} a_{t-s}$, 
converges in probability to $m_{0,true}$ as $t \rightarrow \infty$. In 
particular, the remainder term does \textit{not} converge to $0$. 
This is in sharp contrast with the exponential convergence of this term 
to zero in the $W>0$ case, as given by Lemma \ref{lem:state_magnitude_bound}. 

The above example can be generalized as follows:
\begin{example}
In any dimension $n$, let $(G,F)$ define an LDS such that $G$ is a 
rotation, and such that $G,F$ is observable. Again choose $W=0$ and $v>0$.  
As before, let the true process start from a state $m_{0,true} \neq 0$ and 
start the filter with $m_0 = 0$ and $C_0 = \Id$.
\end{example}
Considerations similar to those of the previous example imply that $R_t 
\rightarrow 0$ but $f_{t+1}$ does not. Consequently, the remainder term 
will not converge to zero. 

\section{An Algorithm and Regret Bounds}
\label{sec:algo}

In this section, we %describe the framework for regret-bound analysis,
introduce our prediction algorithm and prove the associated regret bounds. 
Our on-line algorithm maintains a 
state estimate, which is represented by the regression 
coefficients $\theta\in \RR^s$, where $s$ is the regression depth, a parameter of the algorithm.  
At time step $t$, the algorithm first produces a prediction of the 
observation $Y_t$, using the current state $\theta$ and previous 
observations, $Y_{t-1},\ldots,Y_0$. 
Specifically, we will predict $Y_t$ by 
\begin{equation}
\label{eq:y_hat_definition}
\hat{y}_t(\theta) = \sum_{i=0}^{s-1} \theta_{i} Y_{t-i-1}. 
\end{equation}
After the prediction is made, the true observation $Y_t$ is revealed to 
the algorithm, and a loss associated with the prediction is computed. 
Here we consider the quadratic loss for simplicity:  
We define $\ell(x,y)$ as $(x-y)^2$. The loss function at time $t$ will be given by 
\begin{align}
\label{eq:loss}
\ell_t(\theta) := \ell(Y_t,\hat{y}_t(\theta)). 
\end{align}
In addition, the state is updated. 
We use the general scheme of 
on-line gradient decent algorithms, \cite{Zinkevich2003}, where the update goes against the direction of the gradient of the current loss. In addition, 
it is useful to restrict the state to a bounded domain.
We will use 
a Euclidean ball of radius $D$ as the domain, where $D$ is a parameter of 
the algorithm. We denote this domain by 
$\mathcal{D} = \Set{x\in \RR^s \setsep |x| \leq D}$ and denote by $\pi_{\mathcal{D}}$ the Euclidean projection 
onto this domain. 
If the gradient step takes the state outside of the domain, the state is 
projected back onto $\mathcal{D}$. 
The pseudo-code is presented in Algorithm \ref{alg:OGD}, where the
gradient $\nabla_{\theta} \ell_t(\theta)$ of the cost at $\theta$  at time $t$ is given by
\begin{align}
\label{eq:nabla}
& -2\Brack{Y_t - \sum_{i=0}^{s-1} \theta_i Y_{t-i-1}}
\Brack{Y_{t-1},Y_{t-2},\ldots, Y_{t-s}}.
\end{align}

%\begin{algorithm}[tb]
\begin{algorithm}
   \caption{On-line Gradient Descent}
   \label{alg:OGD}
\begin{algorithmic}[1]
   \STATE {\bfseries Input:} Regression length $s$, domain bound $D$. \\
          \spaceo \spaceo Observations $\Set{Y_t}_0^{\infty}$, 
          given sequentially. \\
   \STATE Set the learning rate $\eta_t = t^{-\half}$. 
   \STATE Initialize $\theta_s$ arbitrarily in $\mathcal{D}$. 
   \FOR {$t = s$ {\bfseries to}  $\infty$}
   \STATE  Predict $\hat{y}_t = \sum_{i=0}^{s-1} \theta_{t,i} Y_{t-i-1}$
   \STATE  Observe $Y_t$ and compute the loss $\ell_t(\theta_t)$  of \eqref{eq:loss}
   \STATE  Update $\theta_{t+1} \gets \pi_{\mathcal{D}}\Brack{ 
                           \theta -  \eta_t \nabla \ell_t(\theta_t) }$ using \eqref{eq:nabla}
   \ENDFOR
\end{algorithmic}
\end{algorithm}

Note a slight abuse of notation in 
Algorithm \ref{alg:OGD}: the vector 
$\theta_t\in \RR^s$ denotes the state at time $t$, while 
in \eqref{eq:y_hat_definition} and elsewhere in the text, $\theta_i$ 
denotes the scalar coordinates of $\theta$. Whether the vector or the 
coordinates are considered will always be clear from context. 

For any LDS $L$, let $f_t(L)$, defined by (\ref{eq:base_prediction}), be 
the prediction of $Y_t$ that Kalman filter associated with $L$ makes, 
given $Y_{t-1},\ldots,Y_0$. We start all filters with the initial state 
$m_0 =0$, and initial covariance $C_0 = \Id_s$, the $s\times s$ identity matrix. 
Let $S$ be any family of LDSs. Then for 
any sequence $\Set{Y_t}_{0}^T$, the quantity
\begin{equation}
\sum_{t=0}^T \ell(\theta_t) - 
\min_{L \in S} \sum_{t=0}^T \ell(Y_t,f_{t}(L)),
\end{equation}
where $\theta_t$ are the 
sequence of states produced by Algorithm \ref{alg:OGD}, is called the 
\textit{regret}. As discussed in the introduction, 
$\sum_{t=0}^T \ell(\theta_t)$ is the total error incurred by the 
algorithm, and $\min_{L \in S} \sum_{t=0}^T \ell(Y_t,f_{t}(L))$ is the 
loss of the best (in hindsight) Kalman filter in $S$. Therefore, small regret 
means that the algorithm performs on sequence $\Set{Y_t}_{0}^T$ as well 
as the best Kalman filter in $S$, even if we are allowed to select that 
Kalman filter in hindsight, after the whole sequence is revealed. 

In the Supplementary Material, we prove 
the following bound on the regret of Algorithm~\ref{alg:OGD}:
\begin{thm} 
\label{thm:regret_bound}
Let $S$ be a finite family of LDSs, such that every 
$L=L(F,G,v,W) \in S$, is observable and has $W>0$. Let $B_0$ be given. For any $\eps>0$,
there are $s$,$D$, and $C_S$, such that the following holds: \\

For every sequence $Y_t$ with $\Abs{Y_t}\leq B_0$, 
if $\theta_t$ is a sequence produced by Algorithm \ref{alg:OGD} with 
parameters $s$ and $D$, then for every $T>0$, 
\begin{equation}
\sum_{t=0}^T \ell_t(\theta_t) - \min_{L \in S} \sum_{t=0}^T \ell(Y_t,f_{t}(L)) \leq 
C_S + 2(D^2 + B_0^2) \sqrt{T} +  \eps T.
\end{equation}
\end{thm}

Due to the limited space in the main body of the text, 
we describe only the main ideas of the proof here.  
Similarly to other proofs in this domain, it consists of two steps. In the 
first step we show that 
\begin{equation}
\label{eq:regret_sketch_st1}
\sum_{t=0}^T \ell_t(\theta_t) - \min_{\phi \in \mathcal{D}} \sum_{t=0}^T 
\ell(Y_t,\hat{y}_t(\phi)) \leq 2(D^2 + B_0^2) \sqrt{T}. 
\end{equation}

\begin{figure*}[tbh!]
%\begin{minipage}{0.31\textwidth}%
%    \centering
\includegraphics[page=102, width=\textwidth]{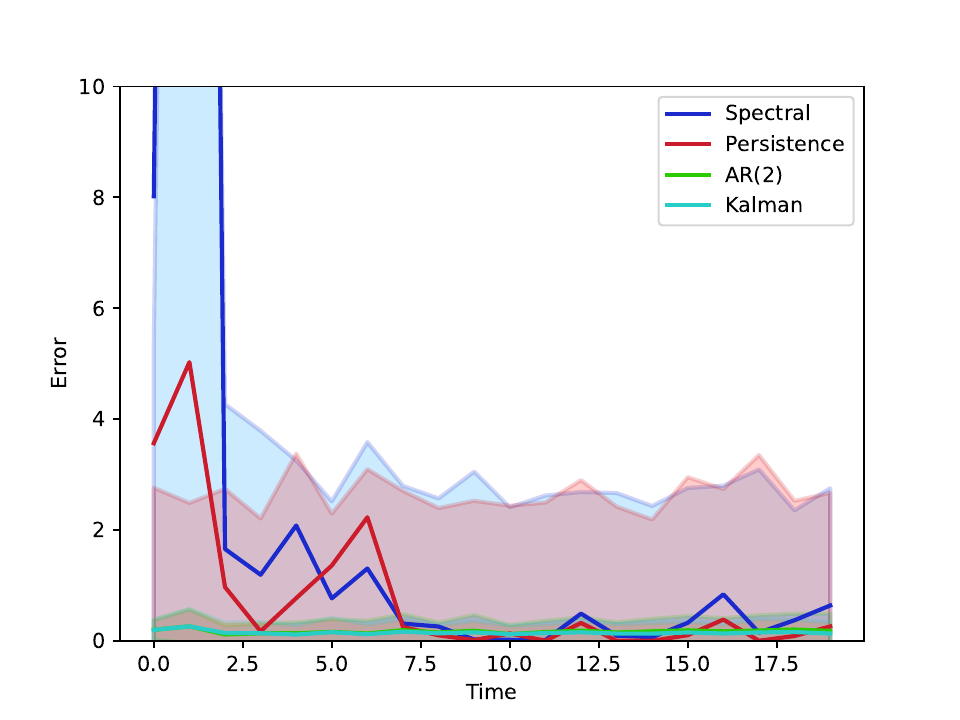}     
    \caption{The error of AR($2$) compared against  Kalman filter, last-value prediction, and spectral filtering in terms of the mean and standard deviation over $N = 100$ runs on Example~\ref{HazanEx}.
    }\label{fig1brief}
%\end{minipage}
\end{figure*}

\begin{figure*}[tbh!]
%\begin{minipage}{0.31\textwidth}%
        \includegraphics[page=3, width=\textwidth]{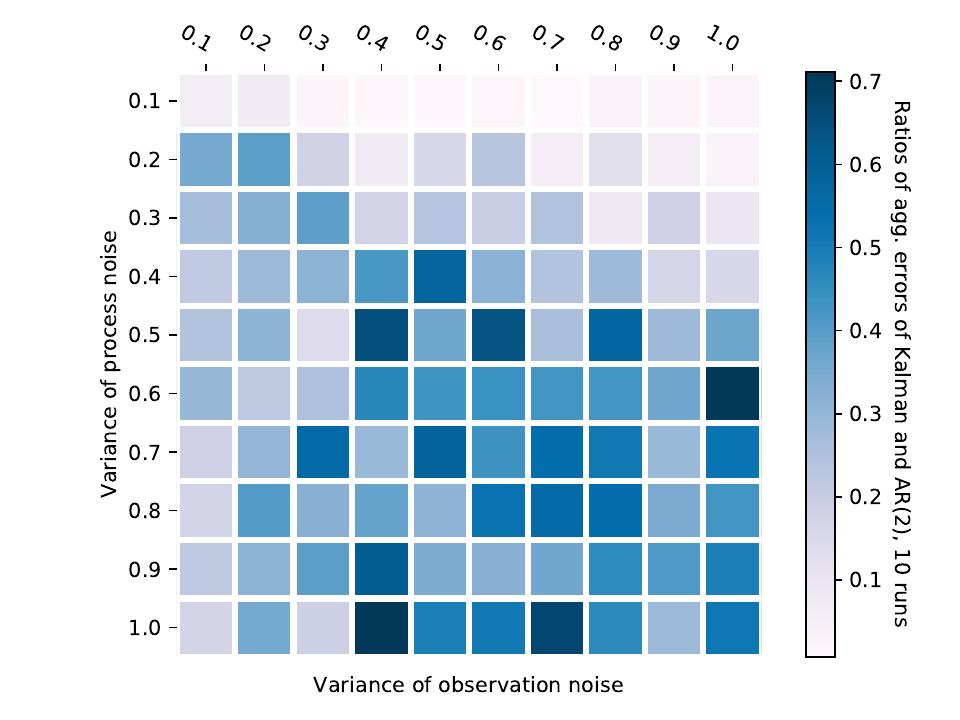}        
\caption{The ratio of the errors of Kalman filter and AR($2$) on Example~\ref{HazanEx}
indicated by colours as a function 
of $w, v$ of process and observation noise, on the vertical and horizontal axes, resp. 
Origin is the top-left corner.
    }\label{figNoisebrief}
%\end{minipage}
\end{figure*}

\begin{figure*}[tbh!]
%\begin{minipage}{0.31\textwidth}%
\includegraphics[page=4, width=\textwidth]{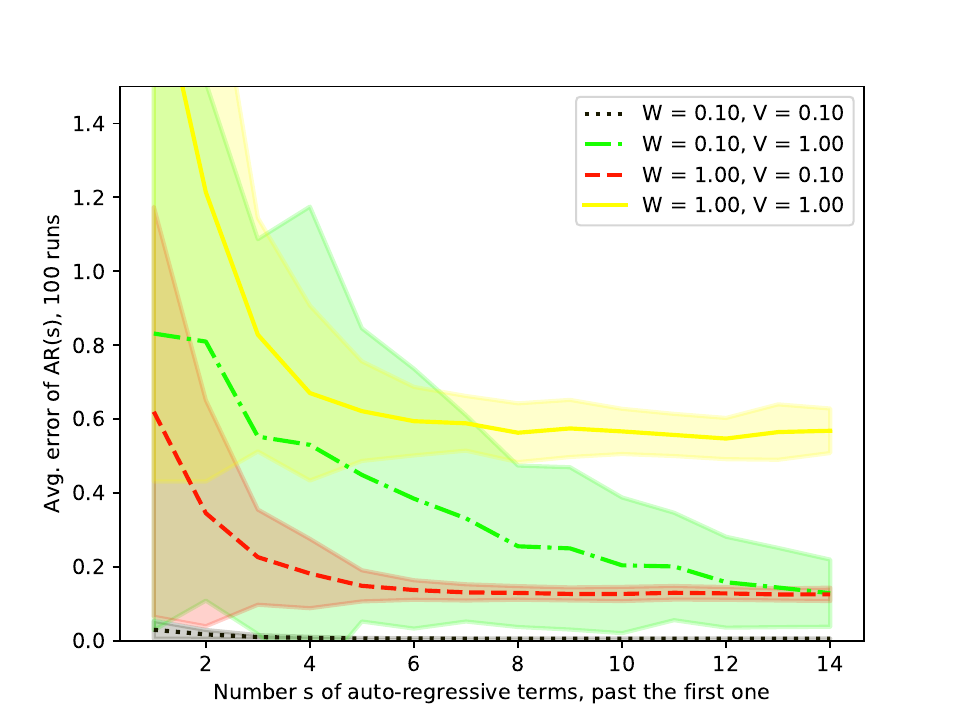}  
\caption{
The error of AR($s + 1$) as a function of $s + 1$,
in terms of the mean and standard deviation 
over $N = 100$ runs on Example ~\ref{HazanEx},
for 4 choices of $W, v$ of process and observation noise, respectively. 
    }\label{figNoise2brief}
%\end{minipage}%
\end{figure*}

This means that Algorithm \ref{alg:OGD} performs as well as the best in 
hindsight \textit{fixed} state vector $\phi$. This follows from the 
general results in \cite{Zinkevich2003}. In the second step, we use the 
approximation Theorem \ref{thm:lds_approximation} to find for each 
$L \in S$ an appropriate $\theta_L \in \mathcal{D}$, such that the 
predictions $f_{t,L}$ are approximated by $\hat{y}_t(\theta_L)$. It 
follows from this step, that the best Kalman filter performs 
approximately as well as the best $\theta_L$. Specifically, we have 
\begin{equation}
\label{eq:regret_sketch_st2}
\min_{L \in S} \sum_{t=0}^T \ell(Y_t,\hat{y}_t(\theta_L))
\leq \min_{L \in S} \sum_{t=0}^T \ell(Y_t,f_{t}(L)) + \eps T
\end{equation}
Because by construction $\theta_L \in \mathcal{D}$, clearly it holds that 
\begin{equation*}
\min_{\phi \in \mathcal{D}} \sum_{t=0}^T 
\ell(Y_t,\hat{y}_t(\phi)) \leq 
\min_{L \in S} \sum_{t=0}^T \ell(Y_t,\hat{y}_t(\theta_L)),
\end{equation*}
and therefore combining (\ref{eq:regret_sketch_st1}) and 
(\ref{eq:regret_sketch_st2}) yields the statement of Theorem 
\ref{thm:regret_bound}. 

\section{Experiments}
\label{sec:experiments}

To illustrate our results, we present experiments on a few well-known 
examples in the Supplementary Material.
Out of those, we chose one to present here:

\begin{example}[Adapted from   \cite{hazan2017online}]
Consider the system:
\begin{align}
\label{eq:experem1_system_hazan}
G = \diag([0.999,0.5]), \quad F' = [1, 1],
\end{align}
with process noise distributed as 
$\omega_t \sim \mathcal{N}(0, w\cdot Id_2)$ and observation noise 
$\nu_t \sim \mathcal{N}(0, v)$ for different choices of $v,w>0$. 
\label{HazanEx}
\end{example}

In Figure~\ref{fig1brief}, we compare the prediction error for 4 methods: 
the standard baseline last-value prediction $\hat{y}_{t+1} := y_t$, also 
known as persistence prediction, the spectral filtering of 
\cite{hazan2017online}, Kalman filter, and AR(2). Here AR(2) is the 
truncation of Kalman filter, given by (\ref{eq:base_prediction}) with 
regression depth $s=1$ and no remainder term. Average error over 100 
observation sequences generated by (\ref{eq:experem1_system_hazan}) with 
$v=w=0.5$ is 
shown as solid line, and its standard deviation is shown a as shaded region.
Note that from some time on, spectral filtering essentially performs 
persistence prediction, since the inputs are zero. Further, note that both 
Kalman filter and AR(2) considerably improve upon the performance of last-
value prediction.

In Figure~\ref{figNoisebrief}, we compare the performance of AR(2) and 
Kalman filter under varying magnitude of noises $v,w$. 
In particular, colour indicates the ratio of the errors of Kalman filter to the errors of AR(2), wherein 
the errors 
are the average prediction error over 10 trajectories 
of (\ref{eq:experem1_system_hazan})
for each cell of the heat-map, with each trajectory of length 50. (The formula is given in the Supplementary Material.)
Consistent with our analysis, one can observe that increasing the variance of process noise 
%leads to a mild increase in the error, but to 
improves the approximation of the Kalman filter by AR(2).

Finally, in Figure~\ref{figNoise2brief}, we illustrate the decay of the 
remainder term by presenting the mean (line) and standard deviation 
(shaded area) of the error as a function of the regression depth $s$. 
There, 4 choices of the covariance matrix $W$ of the process noise and the variance $v$ of the observation noise are considered within Example ~\ref{HazanEx} and the error is averaged over $N = 100$ runs of length $T = 200$.
  Of course, as expected, increasing $s$ 
decreases the error, until the error approaches that of the Kalman filter. 
Observe again that for a given value of the observation noise, the convergence 
w.r.t $s$ is slower for \textit{smaller} process noise, consistently 
 with our theoretical observations. 

\section{Conclusions}
We have presented a forecasting method, which is applicable to arbitrary sequences and comes with a regret bound competing against a class of methods, which includes Kalman filters.

We could generalise the results. 
%First, it is sometimes of interest to consider more general loss functions than the quadratic loss. 
 First, since Theorem 
\ref{thm:lds_approximation} provides approximation in absolute value for 
every large-enough $t$, our regret bounds may be easily extended to other 
losses. Second, for simplicity, we have considered only bounded sequences. 
While this is a standard assumption in the literature, it is somewhat 
restrictive, since, at least theoretically,  LDS observations may grow at a rate of  $\sqrt{t}$. For this reason, we have given the 
approximation theorem also for bounded-difference (Lipschitz) sequences,
and the regret results may be extended to this setting as well. 
One could also provide regret bounds for special cases of LDS, as surveyed by \cite{RG1999}.

\clearpage
\bibliographystyle{aaai}
\bibliography{exp_forgetting}

\clearpage
\newpage

\appendix
\section{Proof Of Theorem \ref{thm:regret_bound}}
\begin{proof}
Let $\eps>0$ be given. For every $L \in S$, let $s(L),\theta(L),T_0(L)$
be the approximation rank, the approximating $\theta$, and convergence  
times for which the approximation (\ref{eq:bounded_approx_in_thm_approx})
holds by Theorem \ref{thm:lds_approximation}. Set $s=\max_{L\in S} s(L)$, 
$T_0 = \max_{L\in S} T_0(L)$, and consider all $\theta(L)$ as vectors in 
$\RR^s$, by padding $\theta(L)$ with zeros if necessary. Then we have 
uniform approximation,
\begin{equation}
\label{eq:thm_regret_uniform_lds_approx}
\Abs{f_t(L) - \sum_{i=0}^{s-1} \theta_i(L) Y_{t-1-i}} \leq \eps
\end{equation}
for all $B_0$ bounded sequences $Y_t$, all $L \in S$ and all $t\geq T_0$. 
We set $D=\max |\theta(L)|$ and apply Algorithm \ref{alg:OGD} with 
parameters $s$ and $D$ to the sequence $Y_t$.

Using the standard results, Theorem 1 of 
\cite{Zinkevich2003}, we have for every $T>0$, 
\begin{equation}
\label{eq:thm_regret_linear_regret}
\sum_{t=0}^T \ell_t(\theta_t) - \min_{\theta \in \mathcal{D}} 
\sum_{t=0}^T\ell_t(\theta) \leq  2(D^2 + B_0^2) \sqrt{T}. 
\end{equation}
These are bounds quantify the performance of Algorithm \ref{alg:OGD} 
against the best $\theta \in D$. We now proceed to compare the loss 
of the best $\theta \in D$ with the loss of the best Kalman filter. 

In what follows we assume $T \geq T_0$. Indeed, since $T_0$ is independent of $T$, 
the sequences are bounded, and the family $S$ is finite, the whole regret 
up to time $T_0$ can be dominated by a constant, $C_{S,1}$. 

It then remains to observe that 
\begin{align}
\label{eq:thm_regret_L_to_D_general}
\min_{L \in S} \sum_{t=0}^T \ell(Y_t,f_{t}(L)) \geq 
\min_{\theta \in \mathcal{D}} \sum_{t=0}^T \ell_t(\theta) -C_S - \eps T. 
\end{align}
To see this, write 
\begin{align*}
\min_{L \in S} \sum_{t=0}^T \ell(Y_t,f_{t}(L)) =  
\min_{L \in S} \sum_{t=0}^{T_0} \ell(Y_t,f_{t}(L)) + 
\sum_{t=T_0}^T \ell(Y_t,f_{t}(L)). 
\end{align*}

First, note that there is a constant $C_{S,2}$ such that 
\begin{equation}
\label{eq:thm_regret_bound_to_T_0}
\max_{L \in S} \sum_{t=0}^{T_0} \Brack{Y_t-\hat{y}_t(\theta_L)}^2 \leq C_{S,2}.
\end{equation}
Again, this follows by the boundedness of the observations and finiteness 
of $S$. 
Next, for every $L \in S$, by (\ref{eq:thm_regret_uniform_lds_approx}) we have
\begin{align}
&\sum_{t=T_0}^T \ell(Y_t,f_{t}(L)) =  
\sum_{t=T_0}^T \Brack{Y_t-f_{t}(L)}^2  \\
&\geq 
\sum_{t=T_0}^T \Brack{Y_t-\hat{y}_t(\theta_L)}^2  -2B_0 \eps (T-T_0) \\
& \geq 
\sum_{t=0}^T \Brack{Y_t-\hat{y}_t(\theta_L)}^2  - C_{S,2} -2B_0 \eps T. 
\end{align}
Therefore, 
\begin{align}
&\min_{L \in S} \sum_{t=0}^T \ell(Y_t,f_{t}(L))  \\
&\geq \min_{L \in S} \sum_{t=0}^T \Brack{Y_t-\hat{y}_t(\theta_L)}^2  
- C_{S,2} -2B_0 \eps T \\
&\geq \min_{\theta \in \mathcal{D}} \sum_{t=0}^T 
\Brack{Y_t-\hat{y}_t(\theta)}^2  - C_{S,2} -2B_0 \eps T. 
\end{align}
Setting $C_S = C_{S,1} + C_{S,2}$ we therefore obtain (\ref{eq:thm_regret_L_to_D_general}). 

Combining (\ref{eq:thm_regret_L_to_D_general}) with 
(\ref{eq:thm_regret_linear_regret}) completes the proof. 
\end{proof}

\section{Proof of Theorem \ref{thm:lds_approximation}, Lipschitz Case}
\begin{proof}
With the notation of the proof of Theorem \ref{thm:lds_approximation} in the bounded case, we now consider the analog of (\ref{eq:lds_approx_thm_abs_bound_bounded_case}) for the Lipschitz case.
In the Lipschitz case we have 
\begin{align}
&\Abs{f_{t+1} - \sum_{j=0}^s Y_{t-j} \theta_j} \leq \nonumber \\
& \Abs{Y_t} \Abs{r_t - \theta_0} + \sum_{j=0}^{s-1} \Abs{Y_{t-j-1}} \Abs{r_{t-j-1} - \theta_{j+1}}   \nonumber \\
& +  (\rho_L')^s c_{1,L}\Brack{\Abs{Y_{t}} + sB_1 + c_{2,L}}. 
\end{align}
Choose $s$ large enough so that $(\rho_L')^s c_{1,L} \leq \delta/2$ and 
$(\rho_L')^s c_{1,L}  \Brack{ sB_1 + c_{2,L}} \leq \eps/2$.  Write 
\begin{equation}
\Abs{Y_{t-j-1}} \leq |Y_t| + (j+1)B_1. 
\end{equation}
Then 
\begin{align}
&\Abs{f_{t+1} - \sum_{j=0}^s Y_{t-j} \theta_j} \leq  \\
\label{eq:thm_approx_l1}
& \Abs{Y_t} \Abs{r_t - \theta_0} + \\
\label{eq:thm_approx_l2}
& \Abs{Y_{t}} \sum_{j=0}^{s-1}  \Abs{r_{t-j-1} - \theta_{j+1}} +  \\
\label{eq:thm_approx_l3}
&\sum_{j=0}^{s-1}  (j+1)B_1\Abs{r_{t-j-1} - \theta_{j+1}} + \\
& \half \delta \Abs{Y_t} + \half \eps. 
\end{align}
Choosing $T_0$ large enough so that for all $t\geq T_0$ we have 
(\ref{eq:thm_approx_l1})+(\ref{eq:thm_approx_l2})+(\ref{eq:thm_approx_l3}) is smaller than 
$\eps/2 + \half \delta \Abs{Y_t}$,
which completes the proof for the Lipschitz 
case.
\end{proof}

\clearpage
\section{Additional Experimental Results}

In this section we present additional experimental results in the  
comparison of different prediction methods.

We first continue the Example \ref{HazanEx} form the main body of the 
paper, with a system given by (\ref{eq:experem1_system_hazan}) and 
$v=w=0.5$. Figure \ref{fig1}(right) shows a sample observations 
trajectory of the system, together with forecast for the four methods. 
Figure \ref{fig1}(left) show the mean and standard deviations of the 
errors for the first 500 time steps. Figure \ref{fig1brief} in the main 
text is the restriction of this Figure \ref{fig1}(left) to the first 20 
steps. Similarly to Figure \ref{fig1brief}, we observe that the AR(2) 
predictions are better than the spectral and persistence methods, and 
worse than the Kalman filter, since only two first terms are considered.

Next, in Figure~\ref{figNoise}, we extend the experiments described in 
Figure \ref{figNoisebrief}. Namely, we compare the performance of AR(2), 
AR(4), and AR(8) against Kalman filter, while consider varying the  
variances $v,w$ of the noise terms. 

On the left, colour indicates the average over 10 trajectories of the root 
mean square error (RMSE) of AR$(s+1)$. In the middle, the colour indicates the average over 10 trajectories
of the difference in absolute value between the RMSE of the AR$(s+1)$ 
and that of Kalman filter.
On the right, the colour indicates the ratio:
\begin{align}
\label{eq:ratioapp}
\frac{
\sum_{i}^{10} \sum_{t}^{T} \ell(Y_{i,t},f_t(L)))
}{
\sum_{i}^{10} \sum_{t}^{T} \ell(Y_{i,t},\hat{y}_{i,t}(\theta_{AR(2)}))
}
\end{align}
where $f_t(L)$ denotes the prediction of the Kalman filter with the ground 
truth system parameters. The sum is over $i=1 \ldots 10$ sample paths 
${Y_{i,t}}$, over time $t = 1 \ldots 50$ for each cell of the heat-map. 
This corroborates the analytical result that process noise leads to an  
increase in the error (left), but also to the increase of the 
approximation ratio (right).

\begin{figure*}[t!]
    \centering
\includegraphics[page=41, width=0.4\textwidth]{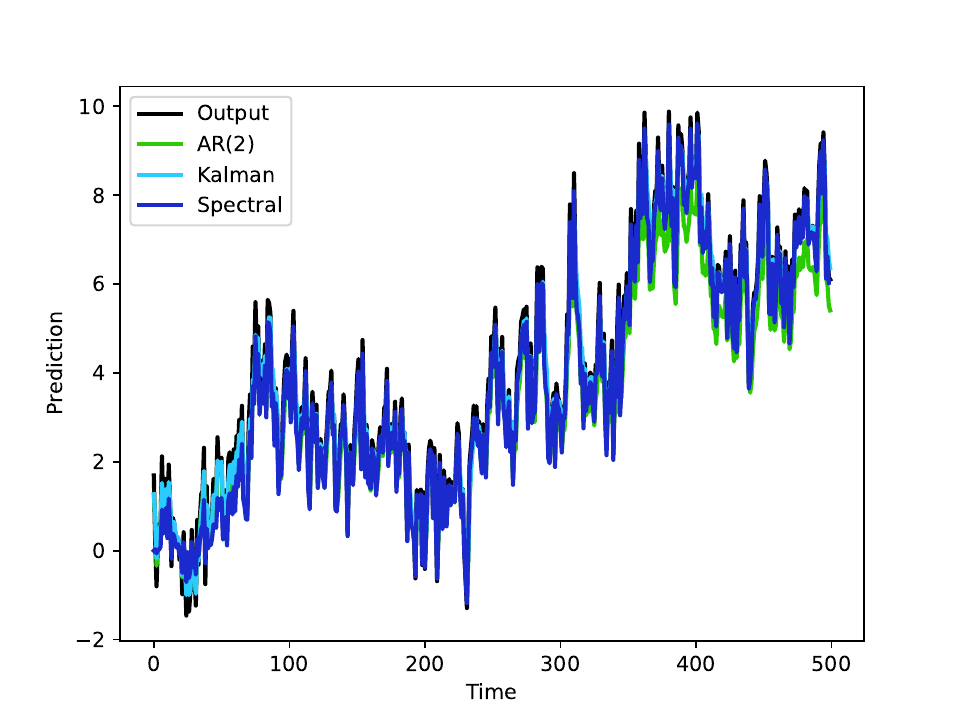}  
\includegraphics[page=101, width=0.4\textwidth]{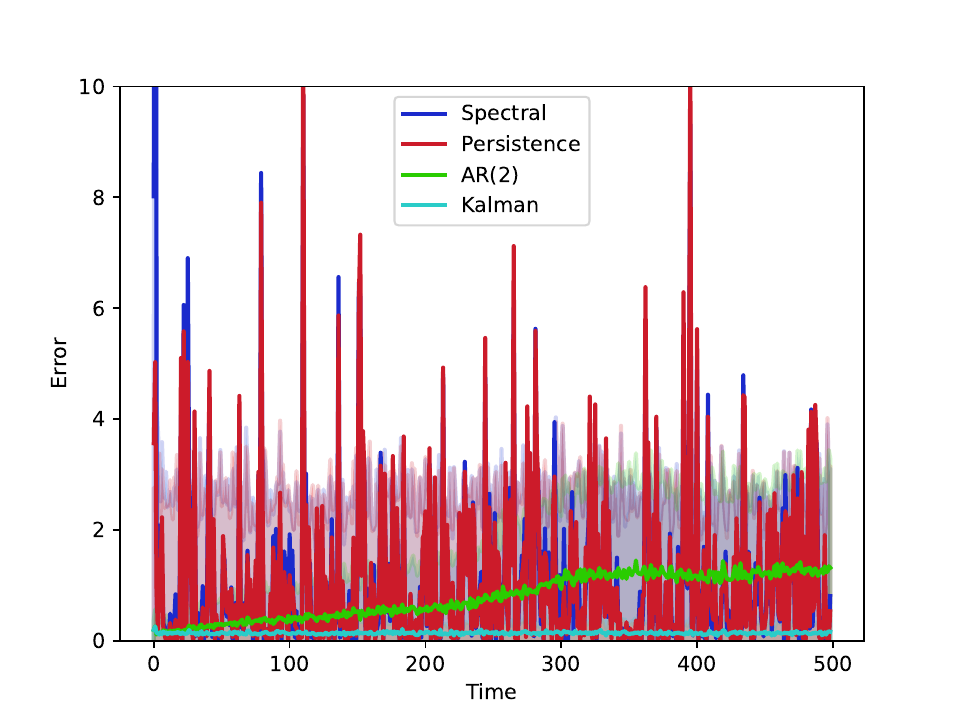}        
    \caption{Illustrations on Example ~\ref{HazanEx}. 
    Left: sample outputs and predictions with AR($2$), compared against Kalman filter, last-value prediction, and spectral filtering of \cite{hazan2017online}.
Right: Same as Figure \ref{fig1brief}, over longer time period. 
    }\label{fig1}
\end{figure*}

\begin{figure*}[t!]
    \centering
 \includegraphics[page=1, width=0.3\textwidth]{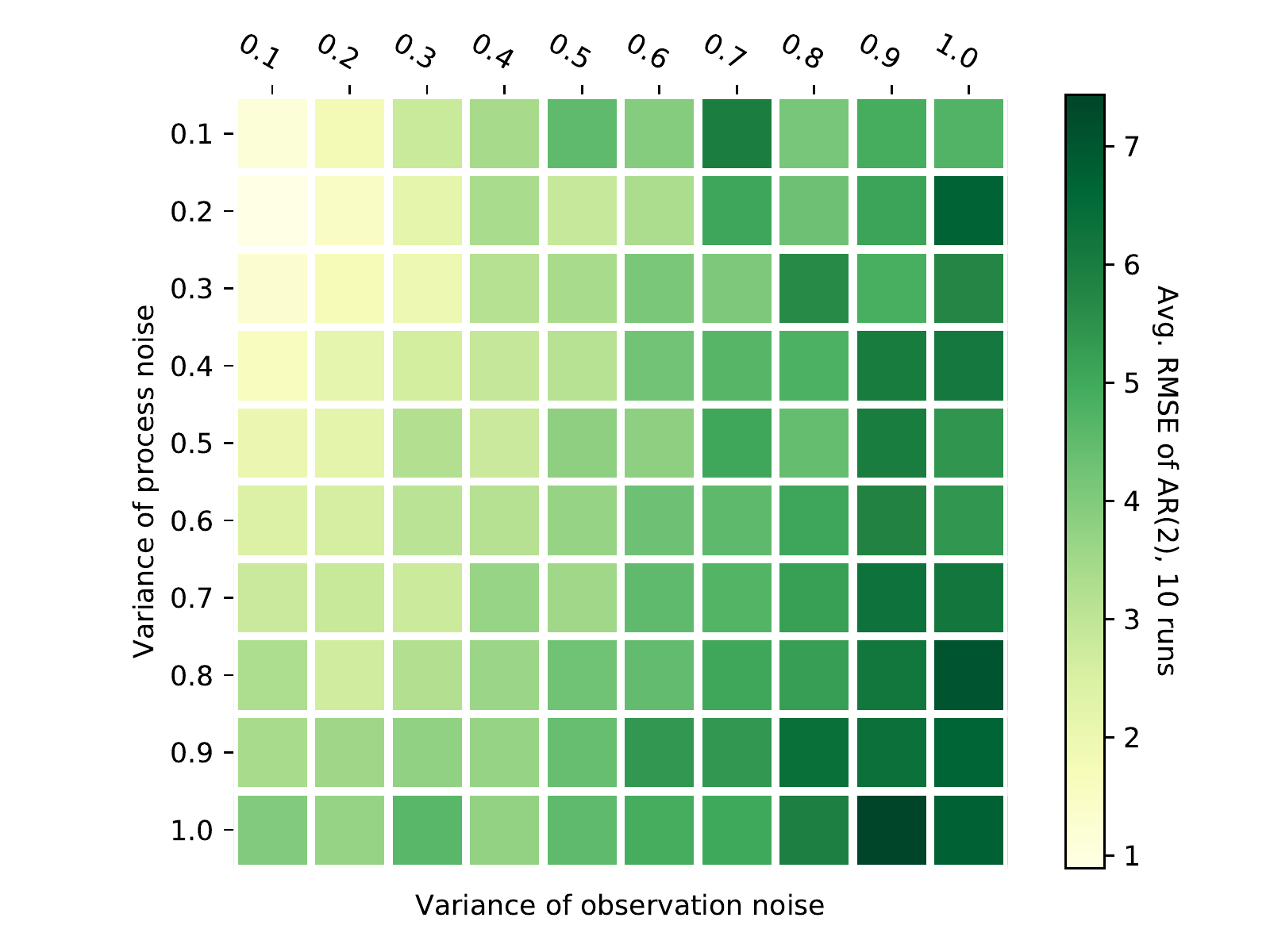}  
        \includegraphics[page=2, width=0.3\textwidth]{noise}   
        \includegraphics[page=3, width=0.3\textwidth]{noise}  \\
 \includegraphics[page=7, width=0.3\textwidth]{noise}  
        \includegraphics[page=8, width=0.3\textwidth]{noise}   
        \includegraphics[page=9, width=0.3\textwidth]{noise}  \\ 
 \includegraphics[page=10, width=0.3\textwidth]{noise}  
        \includegraphics[page=11, width=0.3\textwidth]{noise}   
        \includegraphics[page=12, width=0.3\textwidth]{noise}          
\caption{The effect of varying the  magnitude of noise in Example ~\ref{HazanEx} on AR(2) (top), AR(4) (middle), and AR(8) (bottom). 
Left: average RMSE of predictions of  AR($s+1$) as a function of the variance of the process noise (vertical axis) and observation noise (horizontal axis).
Center: The differences in average RMSE of Kalman filters and AR($s+1$) as a function of the variance of the process noise (vertical axis) and observation noise (horizontal axis). 
Throughout averages are taken over 10 runs.
Right: The ratio \eqref{eq:ratioapp} of the errors of Kalman filters and AR($s+1$) as a function of the variance of the process noise (vertical axis) and observation noise (horizontal axis).  Throughout, notice the origin is in the top-left corner. 
    }\label{figNoise}
\end{figure*}

\iffalse

Next, we  consider another example, admittedly pathological.

\begin{example}
Consider the system:
\begin{align*}
G = \begin{bmatrix} 1.0 & -0.8 \\ -0.6 & 0.3 \end{bmatrix} \quad F' = [1.0, 1.0],
\end{align*}
and noise terms $\eta$ and $\xi$ distributed as $\mathcal{N}(0, 0.5)$. 
\label{MarkEx}
\end{example}

This example is somewhat pathological, because the state-to-state transition matrix is no longer symmetric and its largest eigenvalue is approximately 1.426. This violates the assumptions of many approaches other than ours.
In Figure~\ref{figMark}, one can observe 
that while for some realisation of the noise, the AR(2) still performs well, 
and we have the regret bound against Kalman filters, there may be other methods that
perform better still.
We should like to stress again that this is not a typical example and the typical performance. 

\begin{figure*}[t!]
    \centering
 \includegraphics[page=94, width=0.49\textwidth]{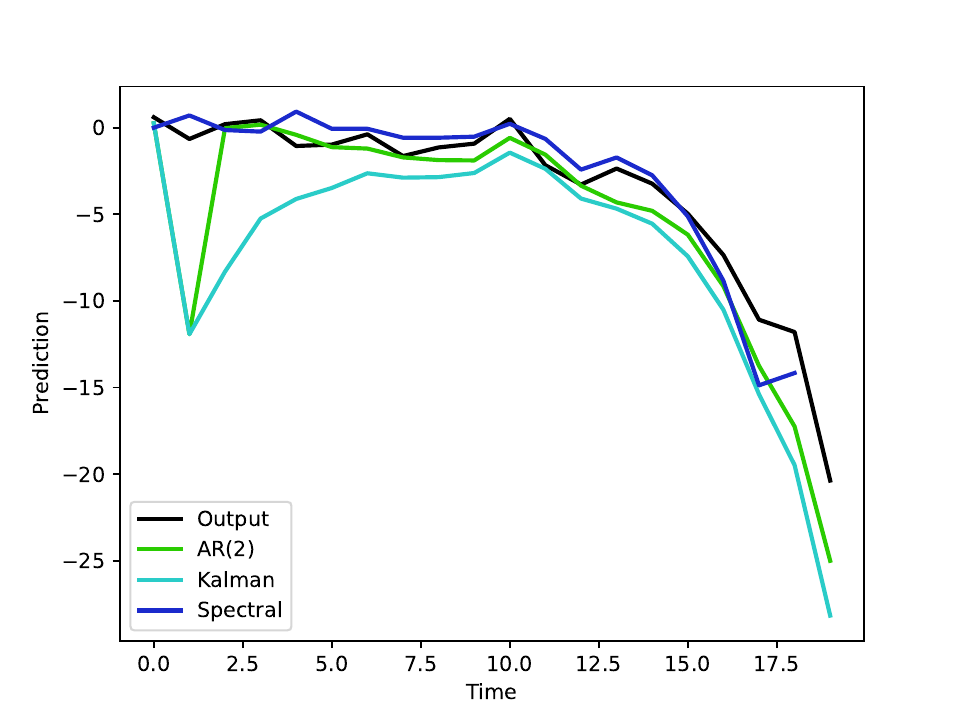}  
\includegraphics[page=102, width=0.49\textwidth]{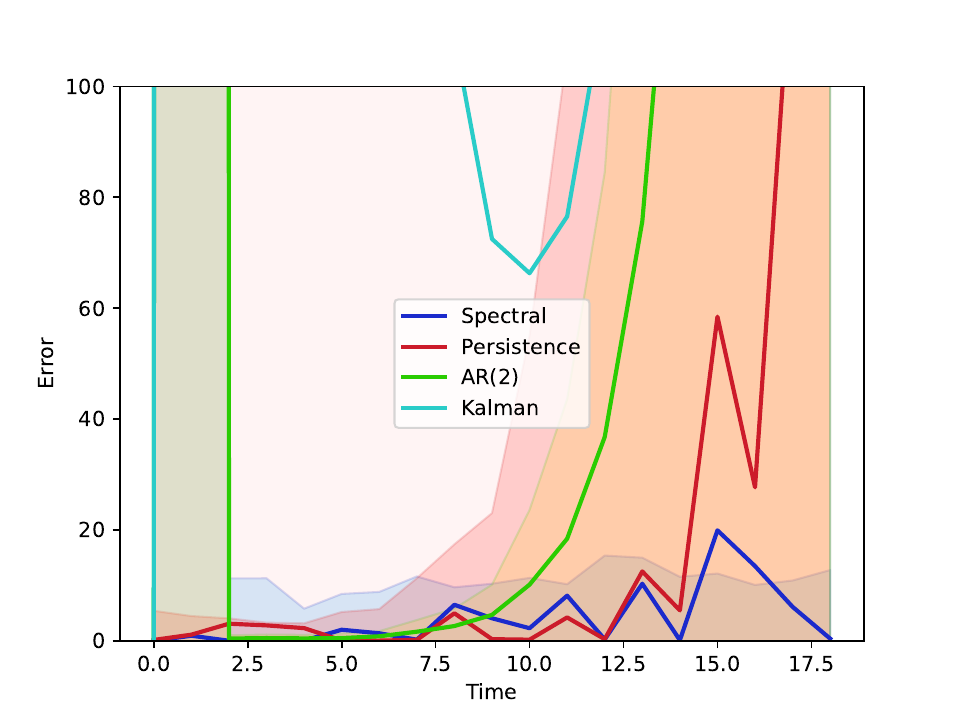}        
    \caption{Illustrations on the pathological Example ~\ref{MarkEx}. 
    Left: sample outputs and predictions with AR($2$), compared against Kalman filter, last-value prediction, and spectral filtering of \cite{hazan2017online}.
Right: Mean and standard deviation over $N = 100$ runs of the same 4 methods.
    }\label{figMark}
\end{figure*}
\fi

Next, we consider a real-world dataset:

\begin{example}[Cited from  \cite{arima_aaai}]
Consider the time-series of daily index of Dow Jones Industrial Average 
for years 1885--1962. There are three versions of the sequence: d0 
represents the index, and d1 and d2 first and second order differencing of 
the series.
\label{realistic}
\end{example}

We compare the performance of the online gradient descent Algorithm 1 to 
that of the spectral filtering of \cite{hazan2017online}. 

It is clear from our earlier discussion of the run-time of the methods (p. 
\pageref{runtimediscussion} in the main body of the paper) that the run-
time of the spectral method grows too fast to handle the complete time 
series.
%We note that in case of \cite{hazan2017online}, we consider $k = 5$ spectral filters and the termination criterion of 10000 iterations of the gradient method, both of which is chosen so as to flatter spectral filtering. 
%In practice, more spectral filters may be applied and the number of iterations required grows with the dimensions of the instances. 
Hence we concentrate on short sub-sequences, where the comparison is can be made.

For the spectral filter, we consider $k = 5$ 
filters, but the performance does not seem to have much impact on the 
performance. For Algorithm \ref{alg:OGD}, we learn an AR(2) model ($s=2$ 
in the notation of Algorithm \ref{alg:OGD}). We use diameter $D=1$
 and learning rate $\eta_t = c/\sqrt{t}$. 
As it turns out, in many settings, the choice $c = 1$
performs well, but especially when there is a clear trend in the 
data, a higher $c$ may be desirable. 
This is illustrated in Figure \ref{figImpactC} on sequence d0. 
As can be seen from the top-left plot, low values of $c$ may lead to slow 
convergence and hence high errors initially.
It may hence be preferable to 
increase the $c$, as illustrated on the right.
We note that for the (essentially stationary) sequences d1 and d2 
from Example ~\ref{realistic}, $c= 1$ works well.

In Figure \ref{fig2} we compare the predictive performance of Algorithm 
\ref{alg:OGD} to spectral filtering and to the persistence (last seen 
value) predictor on the first $T = 100$ elements of the 
three sequences of Example~\ref{realistic}. The errors of spectral 
filtering are two to three orders of magnitude larger than for the last-
value prediction (940 vs. 1.47, 299 vs. 3.58, 4689 vs. 11.02). While some 
of this is due to the large spikes, the errors are non-negligible 
throughout. Algorithm \ref{alg:OGD} performs substantially better than 
either method used for comparison.

\iffalse
Figure \ref{fig2Long} suggests that this behaviour can be observed 
also along longer time-horizons of $T = 200$ (left), $T = 2,000$ (middle), and $T = 20,000$ (right). 
Notice in the plot on the right that the economic troubles after World War I (from position 12,500) and the ensuing Great Depression (until 20,000) led to an increase in the error of AR(2), as is the case for many other methods.
Notice that we do not present the comparison against Kalman filters in this case, where the data have not been generated by an actual linear dynamical system. This is because the error of Kalman filter is essentially determined by the error of the estimate of the matrices, and hence down to the EM algorithm, rather than Kalman filter. 
\fi

\begin{figure*}[t!]
    \centering
 \includegraphics[width=0.3\textwidth,page=1]{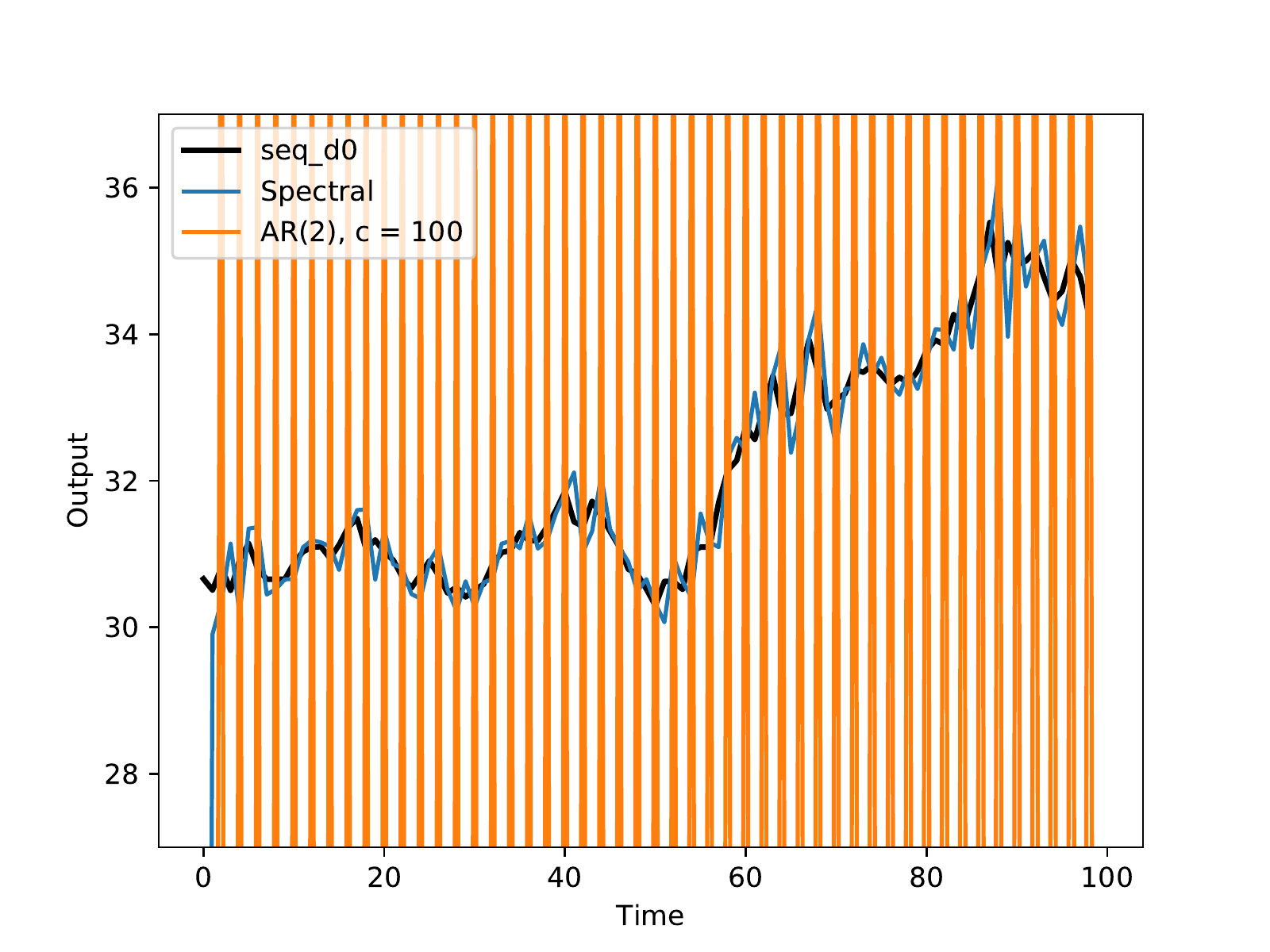}
 \includegraphics[width=0.3\textwidth,page=3]{seq0-short-k5}
 \includegraphics[width=0.3\textwidth,page=5]{seq0-short-k5}\\
  \includegraphics[width=0.3\textwidth,page=2]{seq0-short-k5}
 \includegraphics[width=0.3\textwidth,page=4]{seq0-short-k5}
 \includegraphics[width=0.3\textwidth,page=6]{seq0-short-k5}
\caption{An illustration of the impact of constants in the learning rate on Example ~\ref{realistic}, the well-known time series. Top: The forecasts for three different values of $c$.
Bottom: The error for three different values of $c$.
    }\label{figImpactC}
\end{figure*}

\begin{figure*}[t!]
    \centering
  \includegraphics[width=0.3\textwidth,,page=5]{seq0-short-k5}
\includegraphics[width=0.3\textwidth]{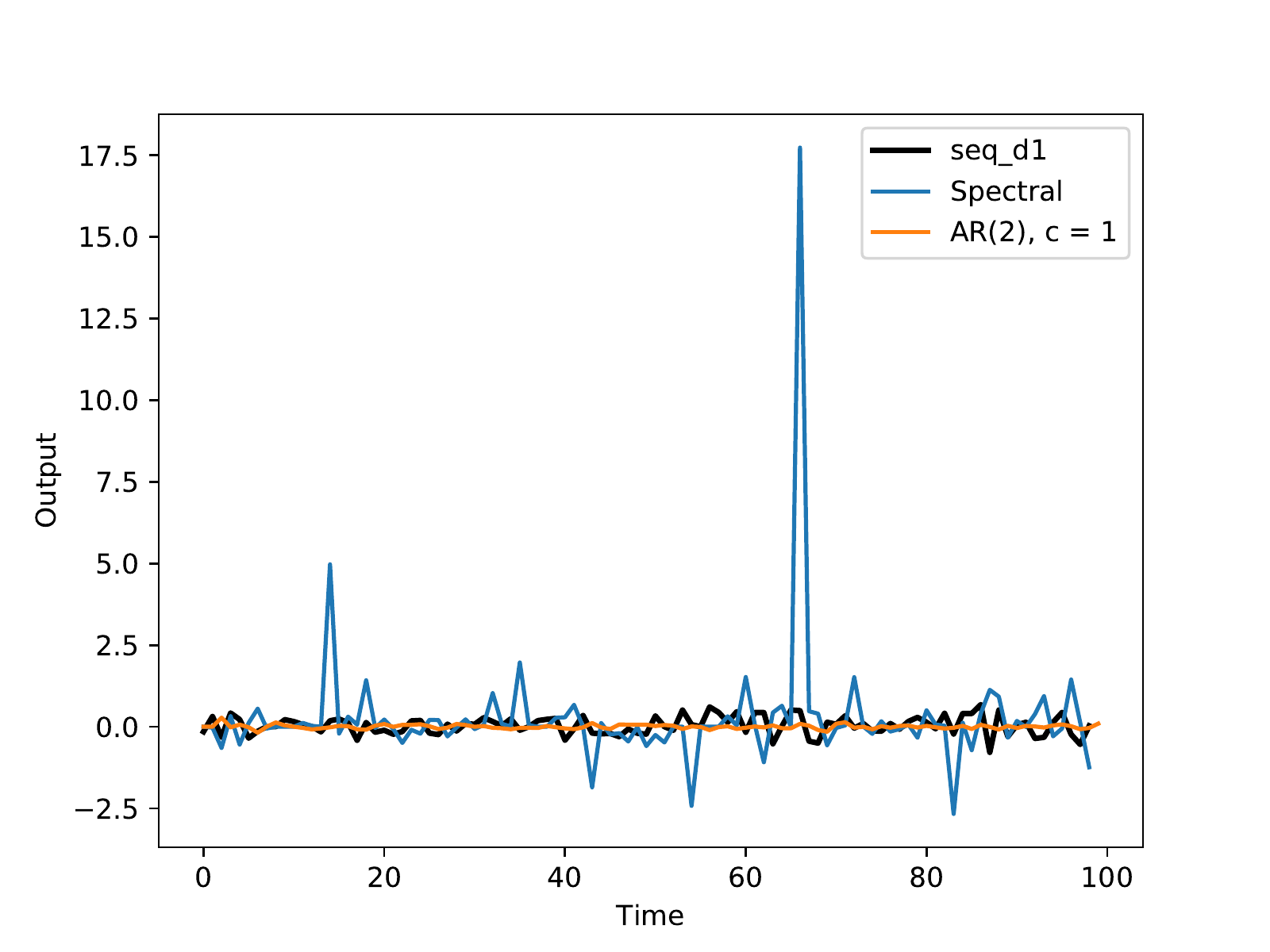}
\includegraphics[width=0.3\textwidth]{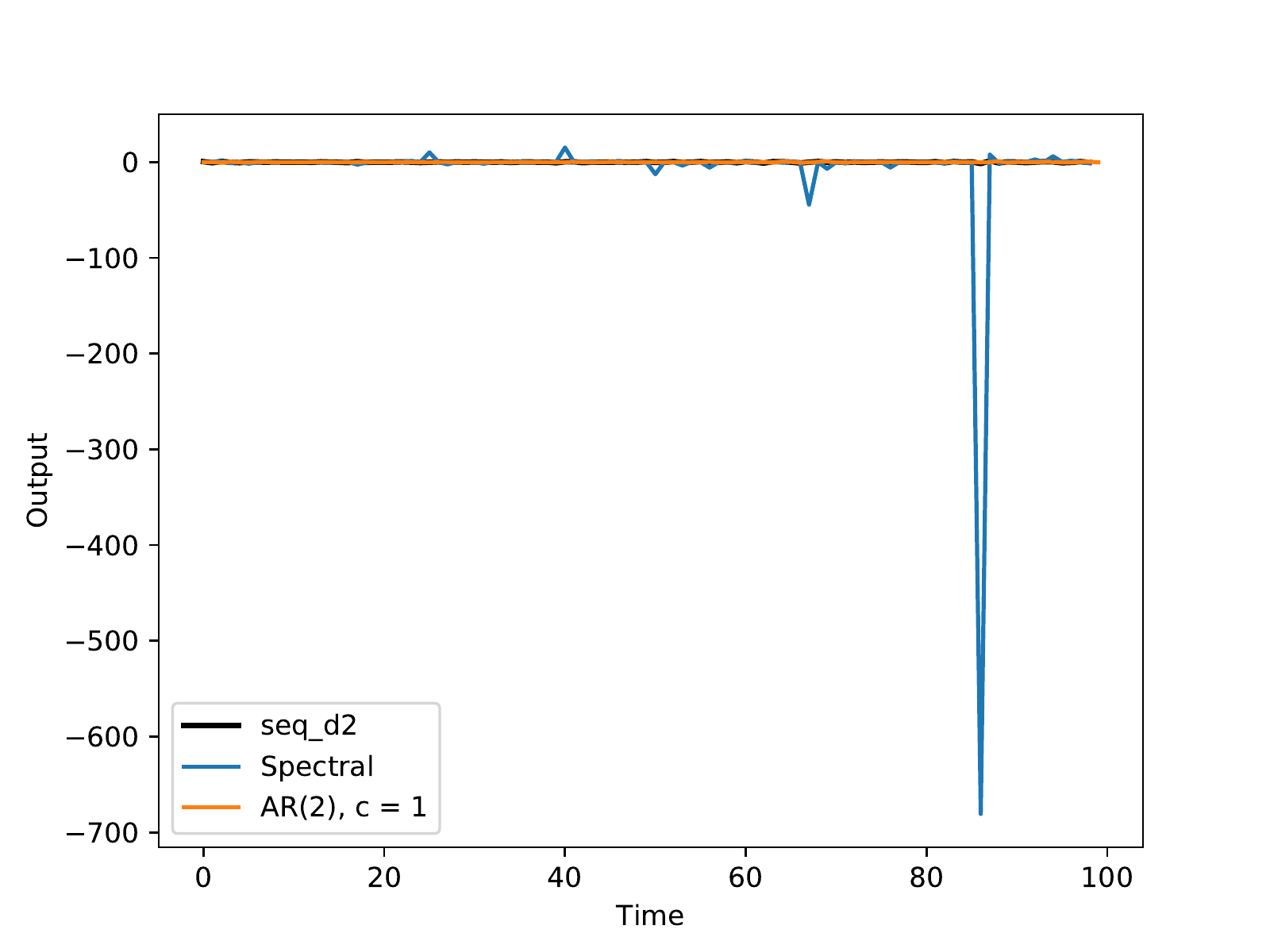} \\
 \includegraphics[width=0.3\textwidth,page=6]{seq0-short-k5}
\includegraphics[width=0.3\textwidth,page=2]{seq1-short-k5}
\includegraphics[width=0.3\textwidth,page=2]{seq2-short-k5}
    \caption{Illustrations on Example ~\ref{realistic}, the well-known time series.
    Top: the predictions of AR$(2)$ compared with the predictions of the spectral filter of \cite{hazan2017online} and the trivial last-value prediction on the first $T = 100$ elements of series d0 (left), d1 (center), and d2 (right).
    Bottom: the corresponding errors.
    }\label{fig2}
\end{figure*}

\end{document}